\documentclass[11pt]{article}

\usepackage{amsmath,amsthm,thmtools,amssymb,amsfonts, graphicx}
\usepackage{graphics}
\usepackage{authblk}
\usepackage{upgreek}
\usepackage{amsrefs}
\usepackage[hypertexnames=false]{hyperref}
\usepackage{cleveref}
\usepackage[left=0.5in,right=1in,top=0.7in,bottom=1in]{geometry}
\usepackage{longtable}
\usepackage{epstopdf,tcolorbox}
\usepackage{empheq}
\usepackage{authblk}
\theoremstyle{plain}
\newtheorem{theorem}{Theorem}[section]

\newtheorem{definition}[theorem]{Definition}

\newtheorem{lemma}[theorem]{Lemma}
\newtheorem{proposition}[theorem]{Proposition}

\theoremstyle{definition}
\newtheorem{remark}[theorem]{Remark}
\numberwithin{equation}{section}
\newcommand{\diff}{\mathop{}\!\mathrm{d}}

\newcommand{\rb}{\bar{r}}
\newcommand{\zb}{\bar{z}}

\title{Global regularity for axisymmetric, swirl-free solutions of the Euler equation in four dimensions}

\author[1]{Evan Miller}
\affil[1]{University of Maine,
Department of Mathematics and Statistics}
\affil[1]{evan.miller1@maine.edu}

\begin{document}

\maketitle

\begin{abstract}
In this paper, we prove global regularity for all smooth, axisymmetric, swirl-free solutions of the incompressible Euler equation in four dimensions. Previous works establishing global regularity for certain axisymmetric, swirl-free solutions of the Euler equation in four dimensions required the additional assumption that 
$\frac{\omega^0}{r^2}\in L^\infty$, which can fail even for Schwartz class initial data. For discussion of another contemporaneous result removing this condition, see \Cref{SWZremark}. The key advance in this paper is a new bound on the vortex stretching term that only requires
$\frac{\omega^0}{r^2}\in L^{2,1}(\mathbb{R}^4)$, a condition which holds generically for any axisymmetric, swirl-free initial data $u^0\in H^s\left(\mathbb{R}^4\right), s>4$, with reasonable decay at infinity.
\end{abstract}

\section{Introduction}

The incompressible Euler equation is a fundamental equation in fluid mechanics, describing the motion of an idealized fluid with no viscosity. The Euler equation can be expressed as
\begin{align}
\partial_t u+ (u\cdot\nabla)u+\nabla p
&=0 \\
\nabla \cdot u&= 0,
\end{align}
where $u\in \mathbb{R}^d$ is the velocity and $p$ is the pressure. The first equation expresses Newton's second law, $\Vec{F}=m\Vec{a}$, and the second equation expresses conservation of mass. Using the Helmholtz decomposition, the evolution equation can be expressed in terms of the projection onto divergence free vector fields as
\begin{equation}
    \partial_t u
    +\mathbb{P}_{df}(u\cdot\nabla)u
    =0.
\end{equation}

An important quantity in the dynamics of the Euler equation is the vorticity, which mathematically is the anti-symmetric part of the velocity gradient $A_{ij}=\frac{1}{2}\left(\partial_iu_j-\partial_j u_i\right)$, and physically describes the rotation induced by the fluid flow. In two dimensions, the vorticity can be expressed as the scalar function $\omega=\partial_1 u_2-\partial_2 u_1$, and in three dimensions as the vector-valued function $\omega=\nabla \times u$.

In two dimensions, the vorticity is advected by the flow with 
\begin{equation}
    \partial_t \omega+ (u\cdot \nabla)\omega=0.
\end{equation}
This leads to global regularity, which was first proven by Wolibner \cite{Wolibner} and Hölder \cite{Holder}.

In three dimensions, in addition to advection, there is also vortex stretching with
\begin{equation}
\partial_t\omega 
+(u\cdot\nabla)\omega
-(\omega\cdot\nabla)u
=0.
\end{equation}
This equation can also be expressed equivalently as
\begin{equation}
\partial_t\omega 
+(u\cdot\nabla)\omega
-S\omega
=0,
\end{equation}
where $S$ is the symmetric part of the velocity gradient, $S_{ij}=\frac{1}{2}(\partial_iu_j+\partial_ju_i)$.
It is the difficulty of controlling the vortex stretching term $S\omega$ that makes global regularity such a challenging open problem in three dimensions. Vortex stretching is also a feature in four and higher dimensions, and indeed is stronger relative to advection than in the three dimensional case. In this case the evolution of the vorticity is given by
\begin{equation}
\partial_t A+(u\cdot\nabla) A
+SA+AS=0.
\end{equation}

One important class of solutions when $d\geq 3$ is axisymmetric, swirl-free solutions. We will say that a vector field is axisymmetric, swirl-free, if
\begin{equation}
    u(x)=u_r(r,z)e_r+u_z(r,z)e_z,
\end{equation}
where 
\begin{align}
    r&=(x_1^2+...+x_{d-1}^2)^\frac{1}{2} \\
    z&= x_d \\
    e_r&= \frac{1}{r}(x_1,...,x_{d-1},0) \\
    e_z&= e_d=(0,...,0,1).
\end{align}
Note that this symmetry is preserved by the dynamics of the Euler equation. In particular, if the initial data $u^0\in H^s_{df}\left(\mathbb{R}^d\right), s>1+\frac{d}{2}$, is axisymmetric and swirl-free, then the solution of the Euler equation $u\in C\left([0,T_{max});
H^s_{df}\left(\mathbb{R}^d\right)\right)$ is also axisymmetric and swirl-free for all $0<t<T_{max}$.
For axisymmetric, swirl-free vector fields, the divergence-free constraint can be expressed by
\begin{equation}
\nabla\cdot u= 
\partial_ru_r +\partial_z u_z
+\frac{d-2}{r}u_r=0,
\end{equation}
and the scalar vorticity is given by
\begin{equation}
\omega(r,z)=\partial_r u_z-\partial_z u_r.
\end{equation}
This is related to the antisymmetric part of 
the velocity gradient $A$ by
\begin{equation}
    A(x)=\frac{1}{2}\omega(r,z)(e_r\otimes e_z-e_z\otimes e_r).
\end{equation}
These results are all essentially classical, but for a derivation in arbitrary dimension $d\geq 3$, see \cite{MillerTsai}.

For axisymmetric, swirl-free solutions of the Euler equation, the vorticity has the evolution equation
\begin{equation} \label{VortEqn}
    \partial_t \omega+(u\cdot\nabla)\omega
    -(d-2)\frac{u_r}{r}\omega
    =0,
\end{equation}
and this evolution equation completely determines the dynamics of the solution, because the velocity can be recovered from the vorticity by
\begin{align} 
    u_r(r,z) \label{RadVelo}
    &=
    \frac{d-2}{\pi}
    \int_{-\infty}^\infty \int_0^\infty 
    \rb^{d-2} (\zb-z) \omega(\rb,\zb)  
    \int_{-1}^1 \frac{\tau\left(1-\tau^2\right)^\frac{d-4}{2}}
    {\left(r^2+\rb^2-2r\rb \tau 
    +(z-\zb)^2\right)^\frac{d}{2}} 
    \diff \tau \diff\rb \diff \zb \\
    u_z(r,z) \label{VertVelo}
    &=
    \frac{d-2}{\pi}
    \int_{-\infty}^\infty \int_0^\infty 
    \rb^{d-2} \omega(\rb,\zb)  
    \int_{-1}^1 \frac{(r \tau-\rb)
    \left(1-\tau^2\right)^\frac{d-4}{2}}
    {\left(r^2+\rb^2-2r\rb \tau 
    +(z-\zb)^2\right)^\frac{d}{2}} 
    \diff \tau \diff\rb \diff \zb.
\end{align}
For a detailed derivation, see section 4 in \cite{MillerTsai}.

The vortex stretching term can be absorbed into the advection by dividing the vorticity by $r^{d-2}$, yielding the transport equation
\begin{equation}
    (\partial_t +u\cdot\nabla)
    \frac{\omega}{r^{d-2}}
    =0.
\end{equation}
When $d=3$, the transport equation
\begin{equation}
    (\partial_t+u\cdot\nabla)\frac{\omega}{r}=0,
\end{equation}
yields global regularity for all smooth solutions, which was first proven by Yudovich \cite{Yudovich}.
We will note that for any smooth data, we must have $\frac{\omega^0}{r}\in L^\infty$, because $e_r$ has a singularity at the axis, so $\omega$ must vanish linearly at the axis in order to have $\nabla A\in L^\infty$. This is key to the proof. Interestingly, when $\frac{\omega^0}{r}\notin L^\infty$, Elgindi showed that $C^{1,\alpha}$ solutions of the Euler equation that are axisymmetric and swirl-free can blowup in finite-time \cite{Elgindi}. 
For these solutions, $\omega(r,z)\sim r^\alpha$ with $\alpha \ll 1$ as $r \to 0$.

This opens the door to singularity formation for smooth solutions of the axisymmetric, swirl-free Euler equation when $d\geq 4$, because in this case $\frac{\omega}{r^{d-2}}$ can be unbounded even for Schwartz class vector fields. This makes global regularity a much more challenging problem in four and higher dimensions. Nonetheless, we will prove a very generic global regularity result when $d=4$.

\begin{theorem} \label{GlobalExistenceThmIntro}
    Suppose $u^0 \in H^s_{df}\left(\mathbb{R}^4\right), s>4$ is axisymmetric, swirl-free and that $\frac{\omega^0}{r^2}\in L^{2,1}\left(\mathbb{R}^4\right)$. Then there exists a unique, global smooth solution of the Euler equation $u\in C\left([0,+\infty); H^s_{df}\left(\mathbb{R}^4
    \right)\right)$, and for all $0<t<+\infty$, this solution is axisymmetric, swirl-free and satisfies
    \begin{equation}
    \|\omega(\cdot,t)\|_{L^\infty}
    \leq 
    \left\|\omega^0\right\|_{L^\infty}
    \exp\left(C\left\|\frac{\omega^0}{r^2}
    \right\|_{L^{2,1}}t\right),
    \end{equation}
    where $C>0$ is an absolute constant independent of $u^0$. Furthermore, the condition $\frac{\omega^0}{r^2}\in L^{2,1}$ holds for all axisymmetric, swirl-free $u^0\in H^s_{df}\left(\mathbb{R}^4\right), s>4$, such that for all $r\geq 0, z\in\mathbb{R}$,
    \begin{equation}
    |\omega^0(r,z)|\leq \frac{C'}{(1+r^2+z^2)^2}.
    \end{equation}
\end{theorem}

    \begin{remark}
    The key element of the proof is a new bound on the vortex stretching term:
    \begin{equation} \label{StretchingBoundIntro}
    \left\|\frac{u_r}{r}
    \right\|_{L^\infty
    \left(\mathbb{R}^4\right)}
    \leq C\left\|\frac{\omega}{r^2}\right\|_{L^{2,1}
    \left(\mathbb{R}^4\right)}.
    \end{equation}
    While there is no guarantee that $\frac{\omega^0}{r^2}$ is bounded, even for smooth data, we can guarantee that $\frac{\omega^0}{r^2}\in L^{2,1}\left(\mathbb{R}^4\right)$, for all sufficiently smooth data with reasonable decay at infinity. This norm is strong enough to control the vortex stretching term, and is also preserved by the dynamics of the Euler equation, because $\frac{\omega}{r^2}$ is transported by the flow when $d=4$, and so this leads to global regularity.
    \end{remark}

    \begin{remark}
    The regularity and decay assumptions on the initial data in \Cref{GlobalExistenceThmIntro} are somewhat stronger than absolutely necessary and could be relaxed somewhat. In particular, the rate of decay at infinity could be optimized more, and the bound $\frac{\omega}{r}\in L^\infty$---which comes from the Sobolev embedding $H^{2+\epsilon} \hookrightarrow L^\infty$---could probably be relaxed as well. However, \Cref{GlobalExistenceThmIntro} completely resolves the global regularity problem for axisymmetric, swirl-free solutions of the Euler equation in four dimensions with smooth enough data and moderate decay at infinity.

    All previous global regularity results for axisymmetric, swirl-free solutions of the Euler equation in four dimensions did not apply to generic smooth data with decay at infinity, because they required that $\frac{\omega^0}{r^2}\in L^\infty$, which can fail even for Schwartz class initial data. In particular, Lim and Jeong proved global regularity for axisymmetric, swirl-free solutions of the Euler equation when $d=4$ if $\frac{\omega^0}{r^2}\in L^\infty$ and $\omega^0$ is compactly supported \cite{LimJeongARMA}, giving the bound
    \begin{equation}
    \|\omega(\cdot,t)\|_{L^\infty}
    \leq 
    C_{\omega^0}(1+t)^2,
    \end{equation}
    where $C_{\omega^0}>0$ is a constant depending on the initial data. 
    This improved an earlier global regularity result---also assuming that 
    $\frac{\omega^0}{r^2}\in L^\infty$---proven by Choi, Jeong, and Lim \cite{ChoiJeongLimProc}, but with a bound that is exponential in time, rather than quadratic. This weaker result was also derived by the author and Tsai in \cite{MillerTsai}. Furthermore, Lim proved global regularity in all dimensions $d\geq 3$ when $\frac{\omega^0}{r^{d-2}}\in L^\infty$ under the additional assumption that $\omega^0$ has a single sign \cite{LimSingleSign}. Note that both of these properties are preserved by the dynamics of the Euler equation.

    We emphasize that the improvement from the requirement that $\frac{\omega^0}{r^2}\in L^\infty$ in these previous papers, to the requirement $\frac{\omega^0}{r^2}\in L^{2,1}$ is not merely technical, because the later holds for generic data in $H^s_{df}\left(\mathbb{R}^4\right), s>4$ with reasonable decay at infinity, while the former can fail for even Schwartz class initial data. 
    For example, let
    \begin{equation}
    u(x)
    =
    \exp\left(-r^2-z^2\right)
    \left(r(1-2z^2)e_r
    -z(3-2r^2)e_z\right).
    \end{equation}
    We can then compute that $\nabla\cdot u=0$ and that
    \begin{equation}
    \omega(r,z)=4rz(4-r^2-z^2)
    \exp\left(-r^2-z^2\right),
    \end{equation}
    and so clearly $\frac{\omega}{r^2}\notin L^\infty$.
    For a more detailed discussion, see Proposition 3.23 in \cite{MillerTsai}.
    \end{remark}

    \begin{remark}
        There has recently been a significant amount of research activity focused on the possibility of singularity formation for the Euler equation in very high dimension. While the Euler equation for $d\geq 4$ is not physically realistic, it is nonetheless a mathematically interesting test case for the finite-time blowup problem. The author \cite{MillerInfiniteDimEuler} and Elgindi and Drivas \cite{ElgindiDrivas} both observed blowup of a Burgers shock type when taking the formal infinite dimensional limit of the Euler equation, and these papers considered very different geometries, with the former taking the infinite-dimensional limit of axisymmetric swirl-free solutions, while the latter considered the infinite dimensional limit of solutions on the torus with a strictly upper triangular gradient.

        For axisymmetric, swirl-free solutions, the natural geometry to consider for finite-time blowup is a vorticity that is odd in $z$ and positive in the upper half plane, $z>0$. Physically, this corresponds to colliding vortex rings. This is precisely the geometry in which Elgindi proved finite-time blowup for $C^{1,\alpha}$ solutions in \cite{Elgindi}.
        In this same geometry, Gustafson, Tsai, and the author \cite{GMT} proved a power-law-type lower bound on the growth of vorticity in any dimension $d\geq 3$, which followed on earlier work in dimension $d=3$ by Choi and Jeong \cite{ChoiJeongGrowth}. 
        Hou and Zhang studied the problem of axisymmetric swirl-free solutions of the Euler equation in higher dimensions numerically, and observed more singular behaviour as the dimension increased \cite{HouZhang}. 
        The author and Tsai proved an identity for the evolution of a particular moment of the vorticity in terms of the partition of the kinetic energy into radial and vertical components that heuristically suggests finite-time blowup in very high dimension $d\gg 3$---see Theorem 1.6 in \cite{MillerTsai}.
    \end{remark}

\begin{remark} \label{SWZremark}
    One month after this manuscript was first posted on the arXiv, Shao, Wei, and Zhang published a major breakthrough paper \cite{SWZ} on this problem, proving global regularity for axisymmetric, swirl-free solutions of the Euler equation in $\mathbb{R}^d$ for $4\leq d\leq 6$. Furthermore, they prove global regularity for all $d\geq 4$ if when the domain is a cylinder. The argument is based on control of a very ingeniously chosen Lyapunov functional:
    \begin{align} \label{LyapunovSWZ}
    \Omega_\alpha(r,z,t)
    &=\min(r,R(t))^{d-2-\alpha}
    \left(\frac{\omega(r,z,t)}{r^{d-2}}\right) \\
    &=\begin{cases}
        \frac{\omega(r,z,t)}{r^\alpha}, 
        &r\leq R(t) \\
        R(t)^{d-2-\alpha}\left(\frac{\omega(r,z,t)}
        {r^{d-2}}\right), 
        &r\geq R(t)
    \end{cases},
    \end{align}
    where $0<\alpha<1$, and $R(t)$ goes to zero double exponentially fast.
    This Lyapunov functional exploits special cancellation near the axis $r=0$ that allows the vortex stretching term to be controlled in terms of the $L^{\frac{d}{d-2},\infty}$ norm near the axis in a way that is impossible away from the axis for $d\geq 5$. Away from the axis---that is when $r>R(t)$---$\Omega_\alpha$ is just the transported term multiplying by a decreasing factor, and so control is immediate.

    Given the time required by the peer review process, there is no doubt that Shao, Wei, and Zhang were the first to prove global regularity when $d=4$. Nonetheless, when this paper first appeared on the arXiv, there were no global regularity results in the literature for axisymmetric, swirl-free solutions of the Euler equation with generic initial data in dimension four or higher. In this sense, the two results are contemporaneous for the case $d=4$.

    This paper shows that it is possible to prove global regularity when $d=4$ by a direct, global bound on vortex stretching, whereas for $d=5,6$ the more subtle machinery employed in \cite{SWZ} to control the Lyapunov functional \eqref{LyapunovSWZ} is unavoidable.
     For the case $d=4$, Theorem 1.3 in \cite{SWZ} is slightly stronger than \Cref{GlobalExistenceThmIntro} in this paper in the sense that their result requires the slightly weaker condition $\frac{\omega^0}{r^2}\in L^{2,\infty}$ rather than $L^{2,1}$, but is slightly weaker in that it gives a triple exponential bound on the growth of vorticity, rather than an exponential bound. Their approach is, of course, considerably stronger in the sense that it also yields global regularity when $d=5,6$, which is out of reach of the methods in this paper.
\end{remark}

\begin{remark}
    There are some similarities between the methods used in this paper to those used by Danchin in \cite{Danchin} to prove global regularity for axisymmetric swirl-free solutions of the Euler equation when $d=3$. Danchin proved global regularity when $d=3$ under the assumptions $\frac{\omega}{r}\in L^{3,1}$ and $\omega\in L^\infty \cap L^{3,1}$, which requires less regularity than the initial global regularity proof due to Yudovich, and which, like the result in this paper, has a key bound that involves controlling the vortex stretching term.

    The key difference in approach, beyond the fact that \cite{Danchin} deals with $d=3$ rather than $d=4$, is that the current paper exploits the fact that the Biot-Savart law for axisymmetric, swirl-free solutions has a locally two dimensional structure, even in higher dimension. In particular in two dimensions, we have the Sobolev-type inequality:
    \begin{equation}
    \|f\|_{L^\infty\left(\mathbb{R}^2\right)}
    \leq 
    C\|\nabla f\|_{L^{2,1}\left(\mathbb{R}^2\right)}.
    \end{equation}
    Because $\frac{u_r}{r}$ and $\frac{\omega}{r}$ are both continuous functions and $\omega$ is a derivative of $u$, we might expect a bound of the form 
    \begin{equation} \label{RemarkBoundA}
    \left\|\frac{u_r}{r}\right\|_{L^\infty}
    \leq 
    C \left\|\frac{\omega}{r}
    \right\|_{L^{2,1}(\diff r \diff z)}.
    \end{equation}
    It is straightforward to show that this bound in fact holds from \Cref{VeloBoundProp}, but this ends up being less useful, because it is the physically meaningful measure $r^2\diff r\diff z$, corresponding to the volume in $\mathbb{R}^4$, that is important.
    The bound that we will actually use is
    \begin{equation} \label{RemarkBoundB}
    \left\|\frac{u_r}{r}
    \right\|_{L^\infty}
    \leq C\left\|\frac{\omega}{r^2}\right\|_{L^{2,1}
    \left(r^2\diff r\diff z\right)}.
    \end{equation}
    This is roughly equivalent, because while the right hand sides of the bounds in \eqref{RemarkBoundA} and \eqref{RemarkBoundB} are not equivalent, if we replace the $L^{2,1}$ norm with the $L^2$ norm, then these bounds are exactly equivalent with
    \begin{equation}
    \left\|\frac{\omega}{r^2}
    \right\|_{L^2(r^2\diff r \diff z)}^2
    =
    \int_{-\infty}^\infty \int_0^\infty 
    \left|\frac{\omega(r,z)}{r^2}\right|^2 r^2\diff r\diff z
    =
    \int_{-\infty}^\infty \int_0^\infty 
    \left|\frac{\omega(r,z)}{r}\right|^2 \diff r\diff z
    =
    \left\|\frac{\omega}{r}
    \right\|_{L^2(\diff r \diff z)}^2.
    \end{equation}
    The space $L^{2,1}$ is just a slightly stronger version of $L^2$, so it is not surprising then that both \eqref{RemarkBoundA} and \eqref{RemarkBoundB} hold.
\end{remark}

\subsection{Definitions}

We now define the Sobolev space $H^s$, as well as the Lorentz space $L^{p,q}$.
\begin{definition}
    For all $s\geq 0$, we define $H^s\left(\mathbb{R}^4\right)$ as the Hilbert space with the norm
    \begin{equation}
    \|f\|_{H^s}^2
    =
    \int_{\mathbb{R}^4}
    \left(1+4\pi^2|\xi|^2\right)^s
    |\hat{f}(\xi)|^2 \diff \xi.
    \end{equation}
    We define $H^s_{df}\left(\mathbb{R}^4\right)\subset 
    H^s\left(\mathbb{R}^4;
    \mathbb{R}^4\right)$ as the space of divergence free vector fields, where
    \begin{equation}
    H^s_{df}\left(\mathbb{R}^4\right)
    =
    \left\{u\in H^s\left(\mathbb{R}^4;
    \mathbb{R}^4\right):
    \xi\cdot\hat{u}(\xi)=0\right\}.
    \end{equation}
    Note that $\xi\cdot\hat{u}(\xi)=0$ expresses the divergence free constraint $\nabla\cdot u=0$ in Fourier space.
\end{definition}

\begin{definition}
Let $(X,\mu)$ be a measure space.
For all $1\leq p,q<+\infty$, we define the Lorentz space $L^{p,q}(X,\mu)$ as the space with the quasinorm
\begin{equation}
\|f\|_{L^{p,q}}
=
p^\frac{1}{q}\left(\int_0^\infty
\tau^q \mu\left(\left\{x:|f(x)|>\tau
\right\}\right)^\frac{q}{p}
\frac{\diff\tau}{\tau}
\right)^\frac{1}{q},
\end{equation}
and for $q=+\infty$ we define $L^{p,\infty}(X,\mu)$ as the space with the quasinorm
\begin{equation}
\|f\|_{L^{p,\infty}}
=
\sup_{\tau>0}
\left(\tau \mu\left(\left\{x:|f(x)|>\tau
\right\}\right)^\frac{1}{p}\right).
\end{equation}
Note that for all $1\leq p<+\infty$, we have
\begin{equation}
    \|f\|_{L^{p,p}}
    =
    \|f\|_{L^p}
\end{equation}
\end{definition}

We will make use of two major results involving Lorentz spaces. First, we will use the fact that the dual of $L^{2,1}$ is $L^{2,\infty}$, which can be expressed as one case of the generalized H\"older inequality.

\begin{proposition} \label{HolderProp}
There exists a constant $C_H>0$ such that for all $f\in L^{2,1}(X,\mu)$ and for all $g\in L^{2,\infty}(X,\mu)$, 
\begin{equation}
\left|\int_X f(x)g(x) \diff\mu(x) \right|
\leq C_H
\|g\|_{L^{2,1}}\|f\|_{L^{2,\infty}}.
\end{equation}
\end{proposition}

The second result we will use is that the intersection of the Lebesgue spaces $L^p\cap L^r$ includes all of the Lorentz spaces in between.

\begin{proposition} \label{IntersectionProp}
    For all $1\leq p<q<r\leq +\infty$, we have the inclusion
    \begin{equation}
    L^{q,1} \subset L^p \cap L^r.
    \end{equation}
\end{proposition}

\begin{remark}
    The main measure we will consider in this paper is $r^2\diff r \diff z$, with
    \begin{equation}
    \mu(\Omega)=\int_{\Omega} r^2 \diff r \diff z
    \end{equation}
    We will note that this is equivalent up to a factor of $4\pi$ to the Lebesgue measure on $\mathbb{R}^4$, because for axisymmetric functions in cylindrical coordinates in four dimensions we have
    $\diff x=4\pi r^2\diff r\diff z$.
    For this reason we will treat quasinorms $L^{2,1}\left(r^2\diff r\diff z\right)$ and $L^{2,1}\left(\mathbb{R}^4\right)$ as interchangeable given that
    \begin{equation}
    \left\|\frac{\omega}{r^2}
    \right\|_{L^{2,1}
    \left(\mathbb{R}^4\right)}
    =
    \sqrt{4\pi}\left\|\frac{\omega}{r^2}
    \right\|_{L^{2,1}
    \left(r^2\diff r\diff z\right)}.
    \end{equation}
\end{remark}

We also have the following classical local wellposedness theorem for strong solutions of the Euler equation.

\begin{theorem} \label{LocalWpThm}
    For all $u^0\in H^s_{df}\left(\mathbb{R}^d\right)$ with $d\geq 2, s>1+\frac{d}{2}$, there exists a unique, strong solution of the Euler equation
    $u\in C\left([0,T_{max}); 
    H^s_{df}\left(\mathbb{R}^d\right)\right)$, where
    \begin{equation}
    T_{max}\geq 
    \frac{C_{s,d}}{\left\|u^0\right\|_{H^s}},
    \end{equation}
    and $C_{s,d}>0$ is an absolute constant independent of $u^0$ depending only on $s$ and $d$.
    Furthermore, if $T_{max}<+\infty$, then
    \begin{equation} \label{BKM}
    \int_0^{T_{max}} 
    \|A(\cdot,t)\|_{L^\infty} \diff t
    =+\infty.
    \end{equation}
    Note that for axisymmetric, swirl-free solutions, this regularity criterion can be expressed as
    \begin{equation}
    \int_0^{T_{max}} 
    \|\omega(\cdot,t)\|_{L^\infty} \diff t
    =+\infty.
    \end{equation}
\end{theorem}

\begin{remark}
    The local wellposedness result was proven by Kato in \cite{KatoLocalWP} and the Beale-Kato-Majda regularity criterion \eqref{BKM} was proven by Beale, Kato, and Majda for $d=3$ in \cite{BKM} and by Kato and Ponce for $d\geq 4$ in \cite{KatoPonce}.
    Note that it is common to denote the anti-symmetric part of $\nabla u$ by $\omega$ when $d\geq 4$, and the vorticity cannot be expressed by a vector. We use $A$ instead, because in our context this makes it clearer when we express the relationship between $A$ and the scalar vorticity $\omega(r,z)=\partial_r u_z-\partial_zu_r$.
\end{remark}

\section{Bound on vortex stretching}

In this section, we will prove the bound on the vortex stretching term in \eqref{StretchingBoundIntro}, which is the crux of the global regularity argument.
The proof is broken down into a number of propositions and lemmas for ease of reading.

\begin{proposition} \label{RadVeloProp}
    Suppose $u\in H^s\left(\mathbb{R}^4\right), s>3$ is axisymmetric and swirl-free. 
    Then for all $r\geq 0, z\in\mathbb{R}$,
    \begin{equation}
    u_r(r,z)
    =
    \frac{1}{\pi}
    \int_{-\infty}^\infty \int_0^\infty 
    \frac{\rb^2 (\zb-z)}
    {\left(r^2+\rb^2+(\zb-z)^2\right)^2}
    H\left(\frac{2r\rb}
    {r^2+\rb^2+(\zb-z)^2}\right)
    \omega(\rb,\zb)  
    \diff\rb \diff \zb,
    \end{equation}
    where for all $0\leq \eta<1$,
    \begin{equation}
    H(\eta)=\int_{-1}^1
    \frac{\tau}{(1-\eta\tau)^2} \diff\tau
    \end{equation}
\end{proposition}

\begin{proof}
    This follows immediately from \eqref{RadVelo} by pulling a factor of $\frac{1}{\left(r^2+\rb^2+(\zb-z)^2\right)^2}$ from the integral with respect to $\tau$.
\end{proof}

\begin{proposition} \label{HBoundProp}
    For all $0\leq \eta<1, H(\eta)\geq 0$ and 
    \begin{equation}
    H(\eta)\leq \frac{4\eta}{1-\eta}.
    \end{equation}
\end{proposition}

\begin{proof}
    We begin by making the change of variables $\tau \to -\tau$ when $-1<\tau<0$, which yields
    \begin{align}
    H(\eta)
    &=
    \int_0^1
    \left(\frac{\tau}{(1-\eta\tau)^2}
    -
    \frac{\tau}{(1+\eta\tau)^2}\right)
    \diff \tau \\
    &=
    \int_0^1
    \frac{4\eta\tau^2}
    {(1-\eta\tau)^2(1+\eta\tau)^2}
    \diff\tau.
    \end{align}
    Non-negativity is then obvious, and furthermore
    \begin{align}
    H(\eta)
    &\leq 
    4 \int_0^1 \frac{\eta}{(1-\eta\tau)^2} 
    \diff\tau \\
    &=
    4\left(\frac{1}{1-\eta}-1\right) \\
    &=
    \frac{4\eta}{1-\eta}.
    \end{align}
\end{proof}

\begin{proposition} \label{VeloBoundProp}
    Suppose $u\in H^s\left(\mathbb{R}^4\right), s>4$ is axisymmetric and swirl-free. 
    Then for all $r\geq 0, z\in\mathbb{R}$,
    \begin{equation}
    |u_r(r,z)|
    \leq 
    \frac{8r}{\pi}
    \int_{-\infty}^\infty \int_0^\infty 
    \frac{1}{\left(r^2+\rb^2
    +(\zb-z)^2\right)^\frac{1}{2}
    \left((\rb-r)^2+(\zb-z)^2\right)^\frac{1}{2}}
    |\omega(\rb,\zb)| 
    \diff\rb \diff \zb.
    \end{equation}
\end{proposition}

\begin{proof}
    Applying \Cref{HBoundProp} to \Cref{RadVeloProp}, we find that
    \begin{multline}
    |u(r,z)| 
    \leq
    \frac{1}{\pi}
    \int_{-\infty}^\infty \int_0^\infty 
    \left(\frac{\rb^2 |\zb-z|}
    {\left(r^2+\rb^2+(\zb-z)^2\right)^2}\right)
    \left(\frac{8r\rb}
    {r^2+\rb^2+(\zb-z)^2}\right) \\
    \left(\frac{1}{1-\frac{2r\rb}
    {r^2+\rb^2+(\zb-z)^2}}\right)
    |\omega(\rb,\zb)| 
    \diff\rb \diff \zb.
    \end{multline}
    Simplifying, we may compute that
    \begin{align}
    |u_r(r,z)|
    &\leq 
    \frac{8r}{\pi}
    \int_{-\infty}^\infty \int_0^\infty 
    \left(\frac{\rb^3 |\zb-z|}
    {\left(r^2+\rb^2+(\zb-z)^2\right)^2
    \left(r^2+\rb^2+(\zb-z)^2-2r\rb\right)}\right)
    |\omega(\rb,\zb)| 
    \diff\rb \diff \zb \\
    &\leq 
    \frac{8r}{\pi}
    \int_{-\infty}^\infty \int_0^\infty 
    \frac{1}{\left(r^2+\rb^2
    +(\zb-z)^2\right)^\frac{1}{2}
    \left((\rb-r)^2+(\zb-z)^2\right)^\frac{1}{2}}
    |\omega(\rb,\zb)| 
    \diff\rb \diff \zb.
    \end{align}
\end{proof}

\begin{lemma} \label{gLemma}
    For all $a>0, b\in \mathbb{R}$, let
    \begin{equation}
    g_{a,b}(r,z)
    =
    \frac{1}{(a^2+r^2+(z-b)^2)^\frac{1}{2}
    ((r-a)^2+(z-b)^2)^\frac{1}{2}}.
    \end{equation}
    Then for all $a>0, b\in \mathbb{R}$,
    \begin{equation}
    \|g_{a,b}\|_{L^{2,\infty}(r^2\diff r\diff z)}
    \leq
    \sqrt{8 \pi}. 
    \end{equation}
\end{lemma}

\begin{proof}
    We will prove the result for the case $a=1, b=0$, and then obtain the rest by scaling and translation invariance. For convenience, define $g:=g_{1,0}$.
    Observe that
    \begin{align}
    g_{a,0}(r,z)
    &=\frac{1}{a^2}
    \left(\frac{1}{\left(1+\frac{r^2}{a^2}
    +\frac{z^2}{a^2}\right)^\frac{1}{2}}\right)
    \left(\frac{1}
    {\left(\left(\frac{r^2}{a^2}-1\right)+
    \left(\frac{z^2}{a^2}\right)\right)^\frac{1}{2}}\right) \\
    &=
    \frac{1}{a^2}g\left(\frac{r}{a},\frac{z}{a}\right).
    \end{align}
    This is precisely the scale invariance for $L^{2,\infty}\left(r^2\diff r \diff z\right)$---note that this space has the same scale invariance as $L^2\left(\mathbb{R}^4\right)$---and so for all $a>0$,
    \begin{equation}
    \|g_{a,0}\|_{L^{2,\infty}}
    =
    \|g\|_{L^{2,\infty}}.
    \end{equation}
    Furthermore $L^{2,\infty}(r^2\diff r\diff z)$ is invariant under translations in $z$, and so for all $a>0, b\in\mathbb{R}$, 
    \begin{equation}
    \|g_{a,b}\|_{L^{2,\infty}}
    =
    \|g_{a,0}\|_{L^{2,\infty}}
    =
    \|g\|_{L^{2,\infty}}.
    \end{equation}
    Therefore, it suffices to show that $g\in L^{2,\infty}(r^2\diff r \diff z)$.

    To do this we will consider the set
    \begin{equation}
    \Omega_\tau=
    \left\{(r,z):g(r,z)\geq \tau\right\}
    \end{equation}
    separately in the range $0<\tau\leq 1$ and $\tau\geq 1$. Suppose $\tau\geq 1$.
    Then $g(r,z)\geq \tau$ implies that
    \begin{equation}
    (r-1)^2+z^2 \leq \frac{1}{\tau^2}.
    \end{equation}
    Therefore, we can compute that for all $\tau\geq 1$
    \begin{align}
    \mu\{\Omega_\tau\}
    &=
    \iint_{\Omega_\tau} r^2\diff r\diff z \\
    &\leq 
    \iint_{(r-1)^2+z^2\leq \frac{1}{\tau^2}} 
    r^2\diff r\diff z \\
    &\leq 
    4\iint_{(r-1)^2+z^2\leq \frac{1}{\tau^2}} 
    \diff r\diff z \\
    &=
    \frac{8\pi}{\tau^2}.
    \end{align}
    It then follows immediately that
    \begin{equation} \label{BoundaA}
    \sup_{\tau\geq 1}
    \tau \mu(\Omega_\tau)^\frac{1}{2}
    \leq 
    \sqrt{8\pi}.
    \end{equation}

    Now assume $0\leq \tau<1$.
    We define $B_\rho$ by
    \begin{equation}
    B_\rho=
    \left\{(r,z) \in
    \mathbb{R}^+\times\mathbb{R}: 
    r^2+z^2\leq \rho^2\right\},
    \end{equation}
    and let
    \begin{equation}
    \Tilde{\Omega}_\tau=\Omega_\tau\cap B_2^c.
    \end{equation}
    Observe that if $r^2+z^2\geq 4$, then 
    \begin{equation}
    \frac{(r-1)^2+z^2}{r^2+z^2}\geq \frac{1}{4},
    \end{equation}
    and so if $(r,z)\in \Tilde{\Omega}_\tau$, then
    \begin{align}
    \tau
    &\leq
    g(r,z) \\
    &= 
    \frac{1}{(1+r^2+z^2)^\frac{1}{2}
    ((r-1)^2+z^2)^\frac{1}{2}} \\
    &\leq  \frac{2}{r^2+z^2},
    \end{align}
    and consequently
    \begin{equation}
    r^2+z^2 \leq \frac{2}{\tau}.
    \end{equation}
    
    If $\frac{1}{2}\leq \tau<1$, then $\Tilde{\Omega}_\tau=\emptyset$. Therefore, $\Omega_\tau\subset B_2$, and so
    \begin{align}
    \mu(\Omega_\tau)
    &\leq
    \mu(B_2) \\
    &=
    \iint_{\substack{r^2+z^2\leq 4 \\r\geq 0}} r^2 \diff r\diff z \\
    &\leq 
    \pi\int_0^2 \rho^3\diff \rho \\
    &=
    4\pi.
    \end{align}
    We can then compute that
    \begin{equation} \label{BoundB}
    \sup_{\frac{1}{2}\leq \tau<1}
    \tau \mu(\Omega_\tau)^\frac{1}{2}
    \leq 
    \sqrt{4\pi}.
    \end{equation}

    Finally, if $0< \tau\leq \frac{1}{2}$, then $\Tilde{\Omega}_\tau \subset B_{\sqrt{\frac{2}{\tau}}}$, and so $\Omega_\tau \subset B_2 \cup\Tilde{\Omega}_\tau \subset 
    B_{\sqrt{\frac{2}{\tau}}}$.
    Therefore, we may compute that
    \begin{align}
    \mu(\Omega_\tau)
    &\leq 
    \mu(B_{\sqrt{\frac{2}{\tau}}}) \\
    &=
    \iint_{\substack{r^2+z^2\leq \frac{2}{\tau} \\r\geq 0}} r^2 \diff r\diff z \\
    &\leq 
    \pi\int_0^{\sqrt{\frac{2}{\tau}}} \rho^3\diff \rho \\
    &=
    \frac{\pi}{\tau^2}.
    \end{align}
    We can then compute that
    \begin{equation} \label{BoundC}
    \sup_{0<\tau\leq \frac{1}{2}}
    \tau \mu(\Omega_\tau)^\frac{1}{2}
    \leq 
    \sqrt{\pi}.
    \end{equation}
    Putting together \eqref{BoundaA}, \eqref{BoundB}, and \eqref{BoundC}, we find that
    \begin{equation}
    \|g\|_{L^{2,\infty}}
    =
    \sup_{\tau>0} \tau 
    \mu(\Omega_\tau)^\frac{1}{2} 
    \leq 
    \sqrt{8\pi},
    \end{equation}
    which completes the proof.
\end{proof}

\begin{theorem} \label{StretchingBoundThm}
    There exists an absolute constant $C>0$ such that for all $u\in H^s\left(\mathbb{R}^4\right), s>4$, axisymmetric and swirl-free, and satisfying $\frac{\omega}{r^2}\in L^{2,1}(r^2\diff r\diff z)$, we have
    \begin{equation}
    \left\|\frac{u_r}{r}\right\|_{L^\infty}
    \leq C\left\|\frac{\omega}{r^2}\right\|_{L^{2,1}},
    \end{equation}
    where $C>0$ is an absolute constant independent of $u$.
\end{theorem}

\begin{proof}
    We can see from \Cref{VeloBoundProp} that for all $r>0, z\in \mathbb{R}$,
    \begin{equation}
    \left|\frac{u_r}{r}(r,z)\right|
    \leq 
    \frac{8}{\pi}
    \int_{-\infty}^\infty \int_0^\infty 
    g_{r,z}(\rb,\zb)
    \left|\frac{\omega}{r^2}(\rb,\zb)\right| 
    \rb^2 \diff\rb \diff \zb.
    \end{equation}
    Applying \Cref{HolderProp}, the generalized H\"older inequality for Lorentz spaces with exponents $(2,\infty)$ and $(2,1)$, and \Cref{gLemma} we can see that for all $r>0, z\in \mathbb{R}$, 
    \begin{align}
    \left|\frac{u_r}{r}(r,z)\right|
    &\leq 
    C_H\|g_{r,z}\|_{L^{2,\infty}} 
    \left\|\frac{\omega}{r^2}\right\|_{L^{2,1}} \\
    &=
    C_H\|g\|_{L^{2,\infty}} 
    \left\|\frac{\omega}{r^2}\right\|_{L^{2,1}}.
    \end{align}
    This completes the proof, because the set $r=0$ has measure zero, but note that the bound holds for $r=0$ as well. The regularity assumptions on $u$ imply that $\frac{u_r}{r}$ is a continuous function with $\frac{u_r}{r}(0,z):=\partial_r u_r(0,z)$, and so the pointwise bound can be obtained in this case by passing to the limit.
\end{proof}

\section{Global regularity}

In this section, we will show that the Lorentz quasi-norm $\|\frac{\omega}{r^2}\|_{L^{2,1}}$ is a conserved quantity for the Euler equation. When combined with the bound on the rate of vortex stretching given in \Cref{StretchingBoundThm}, this immediately leads to global regularity by providing an exponential bound on the growth of vorticity in $L^\infty$.

\begin{proposition} \label{InvariantProp}
    Suppose $u\in C\left([0,T_{max}); H^s_{df}\left(\mathbb{R}^4\right)\right), s>4$ is an axisymmetric, swirl-free solution of Euler's equation, and that $\frac{\omega^0}{r^2}\in L^{2,1}\left(\mathbb{R}^4\right)$. Then for all $0<t<T_{max}$,
    \begin{equation}
    \left\|\frac{\omega}{r^2}(\cdot,t)
    \right\|_{L^{2,1}}
    =
    \left\|\frac{\omega^0}{r^2}
    \right\|_{L^{2,1}}
    \end{equation}
\end{proposition}

\begin{proof}
    We know that $\frac{\omega}{r^2}$ is transported by the flow with
    \begin{equation} \label{Advect}
    (\partial_t+(u\cdot\nabla))\frac{\omega}{r^2}=0.
    \end{equation}
    The velocity $u$ is divergence free, so the flow map
    \begin{equation}
    \partial_t X(r,z,t)=u(X(r,z,t),t)
    =\big(u_r(X(r,z,t),t),u_z(X(r,z,t),t)\big),
    \end{equation}
     with initial condition
     $X(r,z,0)=(r,z)$,
     is a volume preserving diffeomorphism. The advection equation \eqref{Advect} tells us that
     \begin{equation}
    \frac{\omega}{r^2}\circ X=\frac{\omega^0}{r^2}.
     \end{equation}
     Putting these two facts together, we can conclude that
     \begin{align}
         \mu\left(\left\{(r,z):\left|\frac{\omega}{r^2}
        (r,z,t)\right|
        \geq\tau\right\}\right)
        &=
        \mu\left(\left\{X(r,z,t):\left|\frac{\omega}{r^2}
        (r,z,t)\right|
        \geq\tau\right\}\right) \\
        &=
        \mu\left(\left\{(r,z):\left|\frac{\omega}{r^2}
        \circ X(r,z,t)\right|
        \geq\tau\right\}\right) \\
        &=
        \mu\left(\left\{(r,z):\left|\frac{\omega^0}{r^2}
        (r,z)\right|
        \geq\tau\right\}\right).
     \end{align}
     It then immediately follows that
     \begin{equation}
    \left\|\frac{\omega}{r^2}(\cdot,t)
    \right\|_{L^{2,1}}
    =
    \left\|\frac{\omega^0}{r^2}
    \right\|_{L^{2,1}}.
     \end{equation}
\end{proof}

\begin{theorem} \label{GlobalExistenceThm}
    Suppose $u^0 \in H^s_{df}\left(\mathbb{R}^4
    \right), s>4$ is axisymmetric, swirl-free and $\frac{\omega^0}{r^2}\in L^{2,1}
    \left(r^2\diff r \diff z\right)$. Then there exists a unique, global smooth solution of the Euler equation $u\in C\left([0,+\infty); H^s_{df}\left(\mathbb{R}^4
    \right)\right)$, and for all $0<t<+\infty$, this solution is axisymmetric, swirl-free and satisfies
    \begin{equation}
    \|\omega(\cdot,t)\|_{L^\infty}
    \leq 
    \left\|\omega^0\right\|_{L^\infty}
    \exp\left(C\left\|\frac{\omega^0}{r^2}
    \right\|_{L^{2,1}}t\right),
    \end{equation}
    where $C>0$ is an absolute constant independent of $u^0$.
\end{theorem}

\begin{proof}
    Let $u\in C\left([0,T_{max}); H^s\left(\mathbb{R}^4
    \right)\right)$ be the local smooth solution of the Euler equation, and note that $u(\cdot,t)$ is also axisymmetric and swirl-free for all $0<t<T_{max}$.
    Following streamlines in the vorticity evolution equation
    \begin{equation}
    \partial_t\omega+(u\cdot\nabla)\omega=
    2\frac{u_r}{r}\omega,
    \end{equation}
    immediately implies that for all $0<t<T_{max}$,
    \begin{equation}
    \frac{\diff}{\diff t}
    \|\omega(\cdot,t)\|_{L^\infty}
    \leq 
    2\left\|\frac{u_r}{r}(\cdot,t)
    \right\|_{L^\infty}
    \|\omega(\cdot,t)\|_{L^\infty}.
    \end{equation}
    Applying \Cref{StretchingBoundThm} and \Cref{InvariantProp}, we can compute that
    \begin{align}
    \frac{\diff}{\diff t}
    \|\omega(\cdot,t)\|_{L^\infty}
    &\leq 
    C \left\|\frac{\omega}{r^2}(\cdot,t)
    \right\|_{L^{2,1}}
    \|\omega(\cdot,t)\|_{L^\infty} \\
    &=
    C \left\|\frac{\omega^0}{r^2}
    \right\|_{L^{2,1}}
    \|\omega(\cdot,t)\|_{L^\infty}.
    \end{align}
    Applying Gr\"onwall's inequality, we find that for all $0<t<T_{max}$,
    \begin{equation} \label{ExpBound}
    \|\omega(\cdot,t)\|_{L^\infty}
    \leq 
    \left\|\omega^0\right\|_{L^\infty}
    \exp\left(C\left\|\frac{\omega^0}{r^2}
    \right\|_{L^{2,1}}t\right).
    \end{equation}
    This implies $T_{max}=+\infty$, because the Beale-Kato-Majda criterion requires that if $T_{max}<+\infty$, then
    \begin{equation}
    \int_0^{T_{max}}
    \|\omega(\cdot,t)\|_{L^\infty}
    \diff t
    =
    +\infty,
    \end{equation}
    which is clearly inconsistent with the bound \eqref{ExpBound}. This completes the proof.
\end{proof}

\section{Hypothesis satisfied for generic data}

In this section, we show that hypothesis $\frac{\omega^0}{r^2}\in L^{2,1}$ holds for all smooth solutions---$C^2$ is sufficient---with reasonable decay at infinity. This implies that the global regularity result in \Cref{GlobalExistenceThm} establishes global regularity very generically, and not only for a subset of initial data, as was the case for the previous results when $d=4$.

\begin{proposition} \label{LorentzBoundProp}
    Suppose $u\in H^s_{df}\left(\mathbb{R}^4\right), s>4$ is axisymmetric, swirl-free, and that there exists $C>0$ such that for all $r\geq 0, z\in\mathbb{R}$,
    \begin{equation}
    |\omega(r,z)|
    \leq
    \frac{C}{\left(1+r^2+z^2\right)^2}.
    \end{equation}
    Then $\frac{\omega}{r^2}\in 
    L^{2,1}(r^2\diff r\diff z)$.
\end{proposition}

\begin{proof}
    The Sobolev embedding implies that $u\in C^2\left(\mathbb{R}^4\right)$, that $\nabla A\in C\left(\mathbb{R}^4\right)$ and consequently that $\frac{\omega}{r}\in L^\infty.$ See Proposition 3.22 in \cite{MillerTsai} for details.
    This implies that for all $r\geq 0,z\in\mathbb{R}$
    \begin{equation}
    |\omega(r,z)|<Cr,
    \end{equation}
    and that consequently we have the bounds
    \begin{align}
    \left|\frac{\omega}{r^2}(r,z)\right|
    &\leq 
    \frac{C}{r} \\
    \left|\frac{\omega}{r^2}(r,z)\right|
    &\leq 
    \frac{C}{r^2(1+r^2+z^2)^2}
    \end{align}
    Taking the $\frac{2}{3}$ power of the first bound and the $\frac{1}{3}$ power of the second bound, we can see that for all $r\geq 0, z\in\mathbb{R}$
    \begin{equation}
    \left|\frac{\omega}{r^2}(r,z)\right|
    \leq 
    \frac{C}{r^\frac{4}{3}(1+r^2+z^2)^\frac{2}{3}}.
    \end{equation}

    We will now use this bound to control the $L^\frac{7}{4}$ and $L^\frac{17}{8}$ norms. Observe that
    \begin{align}
    \left\|\frac{\omega}{r^2}
    \right\|_{L^\frac{7}{4}}^\frac{7}{4}
    &\leq 
    \int_{-\infty}^\infty \int_0^\infty
    \frac{C^\frac{7}{4}}{r^\frac{7}{3}(1+r^2+z^2)^\frac{7}{6}}r^2\diff r\diff z \\
    &=
    C^\frac{7}{4}
    \int_{-\infty}^\infty \int_0^\infty
    \frac{1}{r^\frac{1}{3}(1+r^2+z^2)^\frac{7}{6}}
    \diff r\diff z \\
    &<
    +\infty.
    \end{align}
    Likewise observe that
    \begin{align}
    \left\|\frac{\omega}{r^2}
    \right\|_{L^\frac{17}{8}}^\frac{17}{8}
    &\leq 
    \int_{-\infty}^\infty \int_0^\infty
    \frac{C^\frac{17}{8}}{r^\frac{17}{6}(1+r^2+z^2)^\frac{17}{12}}r^2\diff r\diff z \\
    &=
    C^\frac{17}{8}\int_{-\infty}^\infty \int_0^\infty
    \frac{1}{r^\frac{5}{6}(1+r^2+z^2)^\frac{17}{12}}\diff r\diff z \\
    &<
    +\infty.
    \end{align}
    We can therefore conclude that 
    $\frac{\omega}{r^2}\in 
    L^\frac{7}{4}(r^2\diff r\diff z)\cap 
    L^\frac{17}{8}(r^2\diff r\diff z)$,
    and so the inclusion
    \begin{equation} 
    L^\frac{7}{4} \cap L^\frac{17}{8}
    \subset L^{2,1} 
    \end{equation}
    from \Cref{IntersectionProp} completes the proof.
\end{proof}

\begin{remark}
    Our result requires that $\frac{\omega^0}{r^2}\in L^{2,1}$, which  is stronger than control in $L^2$. The key idea is that for the $L^2$ norm, which is just slightly weaker than the $L^{2,1}$ quasi-norm, we have
    \begin{equation} \label{ExampleBound}
    \left\|\frac{\omega^0}{r^2}
    \right\|_{L^2\left(\mathbb{R}^4\right)}^2
    =
    4\pi
    \int_{-\infty}^\infty
    \int_0^\infty 
    \left|\frac{\omega^0}{r^2}\right|^2
    r^2\diff r\diff z
    =
    4\pi \int_{-\infty}^\infty
    \int_0^\infty 
    \left|\frac{\omega^0}{r}\right|^2
    \diff r\diff z.
    \end{equation}
    We know that $\frac{\omega^0}{r}$ must be bounded for smooth enough data, so it is straightforward to see that the integral in \eqref{ExampleBound} is finite. It is then not surprising that the $L^{2,1}$ norm is also finite.
\end{remark}

\begin{remark}
    In this paper we prove global regularity for axisymmetric, swirl-free solutions of the Euler equation in $H^s\left(\mathbb{R}^4\right)$, with $s>4$, although the sharp space for local wellposedness only requires $s>3$. There are two reasons for this. First, this control leads to the bound on $\left\|\frac{\omega}{r^2}\right\|_{L^{2,1}}$ in \Cref{LorentzBoundProp}. Furthermore, because we are working with the vorticity formulation, this is enough regularity to guarantee we have a classical solution to the vorticity equation \eqref{VortEqn}. The regularity assumption here for the initial data is not sharp and could be relaxed somewhat; the goal is to prove global regularity in $d=4$ for generic axisymmetric, swirl-free initial data that are sufficiently smooth.
\end{remark}

\bibliographystyle{plain}
\bibliography{bib}

\end{document}